\documentclass[12pt]{amsart}

\usepackage{amssymb,latexsym}
\usepackage{enumerate}
\usepackage[hyperfootnotes=false]{hyperref}
\usepackage{fullpage}
\usepackage{soul}

\usepackage{fixme}
\fxsetup{
    status=draft,
    author=,
    layout=margin,
    theme=color
}

\makeatletter \@namedef{subjclassname@2010}{%
  \textup{2010} Mathematics Subject Classification}
\makeatother

\newcounter{thm} \numberwithin{thm}{section}
\newtheorem{Theorem}[thm]{Theorem}

\newtheorem*{Problem}{Problem}
\newtheorem{Claim}[thm]{Claim}

\newcommand{\thistheoremname}{}
\newtheorem*{genericthm*}{\thistheoremname}
\newenvironment{namedthm*}[1]
  {\renewcommand{\thistheoremname}{#1}%
   \begin{genericthm*}}
  {\end{genericthm*}}
\usepackage{tikz}
\usetikzlibrary{decorations.pathreplacing}
\tikzset{mybrace/.style={decoration={brace,raise=1.8mm},decorate}}
\tikzset{mybracedown/.style={decoration={brace,mirror,raise=1.8mm},decorate}}
\usetikzlibrary{math}
\usetikzlibrary{patterns}
\usetikzlibrary{calc}

\usetikzlibrary{matrix,backgrounds}
\usetikzlibrary{shapes.geometric}

\author[K. Bhowmick]{Krishnendu Bhowmick} \address{Johann Radon Institute for Computational and Applied Mathematics\\Linz, Austria}
\email{Krishnendu.Bhowmick@oeaw.ac.at}

\date{}

\begin{document}
\baselineskip=17pt

\title{A problem of Erd\H{o}s about rich distances}
\begin{abstract}
    An old question  posed by Erdős asked whether there exists a set of \(n\) points such that \(c \cdot n\) distances occur more than \(n\) times. We provide an affirmative answer to this question, showing that there exists a set of \(n\) points such that \(\lfloor \frac{n}{4}\rfloor\) distances occur more than \(n\) times. We also present a generalized version, finding a set of \(n\) points where \(c_m \cdot n\) distances occurring more than \(n+m\) times. 
\end{abstract}
\maketitle
\section{Introduction}
In a 1997 paper, Erd\H{o}s \cite{p} asked the following question:
\begin{Problem}[Erd\H{o}s]\label{problem 1} For a set of \(n\) points in a plane, can \(c \cdot n\) of the distances occur more than \(n\) times?
\end{Problem}
We provide an affirmative answer to the question by proving the following theorem: 
\begin{Theorem}\label{thm:richdistance}
    For all \(n\in \mathbb{N}\), there exists a set of \(n\) points such that \(\lfloor \frac{n}{4}\rfloor\) distances occur at least \(n+1\) times.
\end{Theorem}

We also show the following generalization of Theorem \ref{thm:richdistance}, indicating that \(c_m\) distances can occur \(n+m\) times.
\begin{Theorem}\label{thm:richdistance2}
    For all \(n\in \mathbb{N}\), there exist a set of \(n \) points such that at least \(\big\lfloor \frac{n}{2(m+1)}\big\rfloor\) distances occur at least \(n +m\) times.
\end{Theorem}

\section{Proof of Theorem \ref{thm:richdistance}}

\begin{figure}[ht]
\centering
\begin{tikzpicture}[scale=1.2]

\def\radius{3}

\foreach \i in {1,2,...,5}
{
    \coordinate (v\i) at (\i*360/5 - 90:\radius);
}
\draw[thick] (v1) -- (v2) -- (v3) -- (v4) -- (v5) -- cycle;

\node[left] at (v1) {$v_1$};
\node[above right] at (v2) {$v_2$};
\node[above left] at (v3) {$v_3$};
\node[left] at (v4) {$v_4$};
\node[below] at (v5) {$v_5$};

\coordinate (v) at (v1);

\foreach \i in {2,3,4,5}
{
    \coordinate (r\i) at ($(v) + (v) - (v\i)$);
}
\draw[dashed, thick] (v1) -- (r2) -- (r3) -- (r4) -- (r5) -- cycle;

\node[below left] at (r2) {$r_2$};
\node[below right] at (r3) {$r_3$};
\node[right] at (r4) {$r_4$};
\node[above] at (r5) {$r_5$};

\end{tikzpicture}
\caption{A set of 9 points with 2 distances appearing 10 times.}
\end{figure}
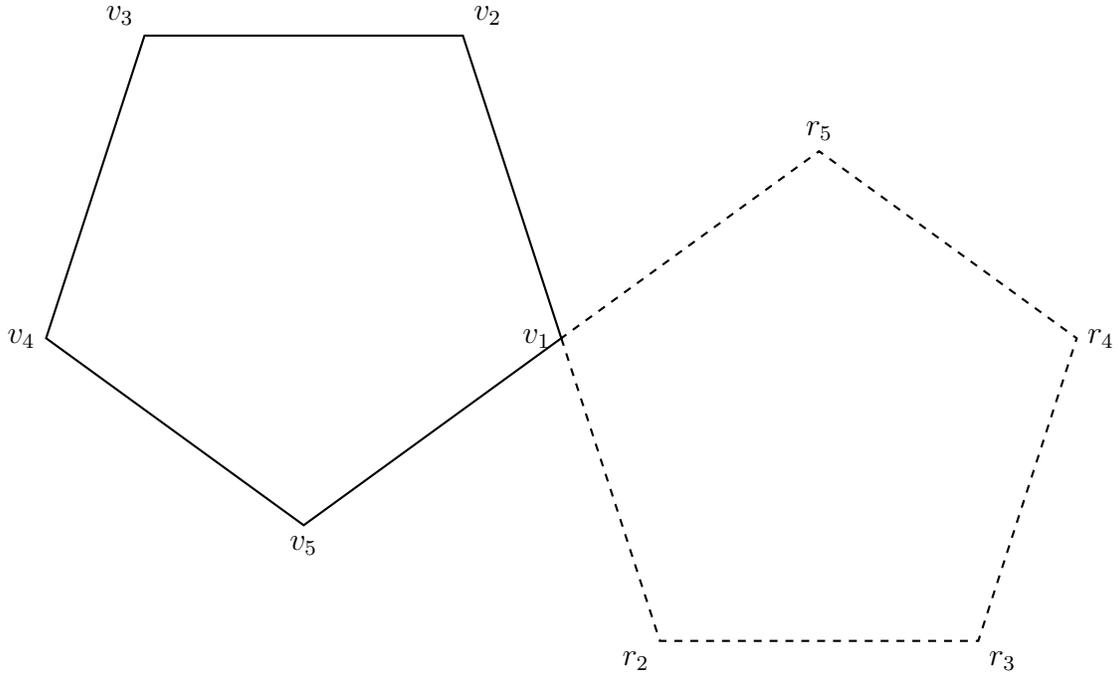

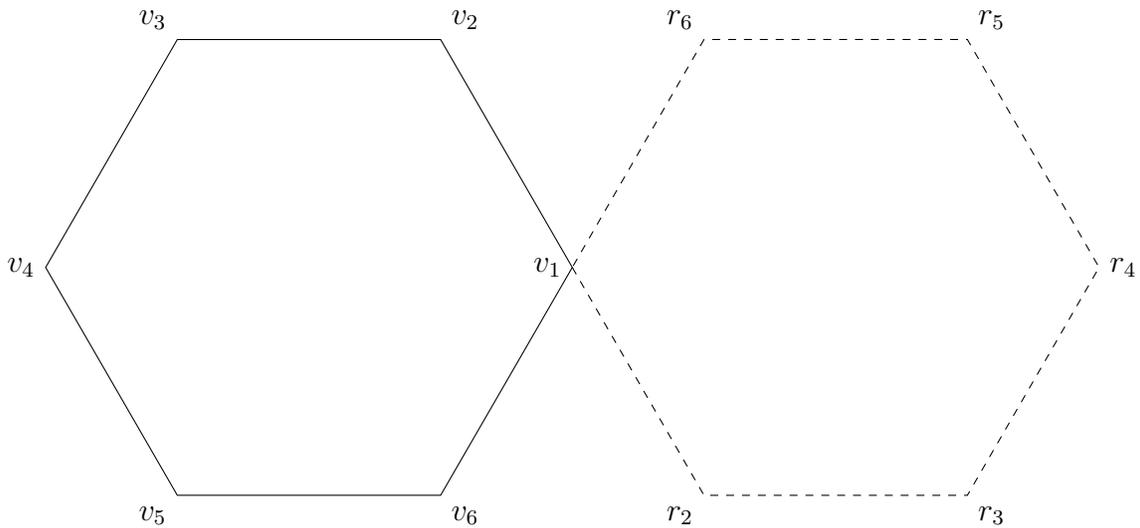
\begin{figure}[ht]
\centering
\begin{tikzpicture}[scale=3.5]

\foreach \i in {1,...,6}
    \coordinate (v\i) at ({60*(\i-1)}:1);

\draw (v1) -- (v2) -- (v3) -- (v4) -- (v5) -- (v6) -- cycle;

\foreach \i in {2,...,6}
    \coordinate (r\i) at ($2*(v1) - (v\i)$);

\draw[dashed] (v1) -- (r2) -- (r3) -- (r4) -- (r5) -- (r6) -- cycle;

\node[left] at (v1) {$v_1$};
\node[above right] at (v2) {$v_2$};
\node[above left] at (v3) {$v_3$};
\node[left] at (v4) {$v_4$};
\node[below left] at (v5) {$v_5$};
\node[below right] at (v6) {$v_6$};

\node[below left] at (r2) {$r_2$};
\node[below right] at (r3) {$r_3$};
\node[right] at (r4) {$r_4$};
\node[above right] at (r5) {$r_5$};
\node[above left] at (r6) {$r_6$};

\end{tikzpicture}
\caption{ A set of 11 points with 2 distances appearing at least 12 times.}
\end{figure}

We start with proving the following simple claim:
\begin{Claim}\label{claim:mgon}
    In a regular \(m-\)gon, \(\lfloor\frac{m-1}{2}\rfloor\) distances appear \(m\) times.
\end{Claim}
\begin{proof} 
Observe that in a regular \(m-\)gon \(v_1\dots v_{m}\), the distances \(\|v_i - v_{(i+k)}\|\) and \(\|v_j - v_{(j+k)}\|\) are equal for all \(i,j \in [m]\) and some \(k\in [\lfloor\frac{m-1}{2}\rfloor]\). Thus, we conclude \(\lfloor\frac{m-1}{2}\rfloor\) of the distances are repeating \(m\) times.
\end{proof}

\begin{proof}[Proof of Theorem ~\ref{thm:richdistance}] 
For \(n < 4\) the statement is vacuously true. Hence, we will assume \(n \geq4\). We will consider two cases: Case 1 for \(n\) odd and Case 2 for \(n\) even.
 
 \textbf{Case 1} Since \(n\) is odd, let \(n = 2m +1\). Consider an \((m+1)-\)gon  \(v_1\dots v_{m+1}\).   From Claim \ref{claim:mgon}, \(\lfloor\frac{m}{2}\rfloor\) of the distances are repeated \(m+1\) times. Now, rotate the \((m+1)-\)gon around vertex \(v_1\) to get a new \((m+1)-\)gon  \(v_1 r_2\dots r_{m+1}\). Again, \(\lfloor\frac{m}{2}\rfloor\) of the distances are repeating \(m+1\) times in the new \((m+1)-\)gon. Since \(v_1\) is the only common vertex between the two \((m+1)-\)gons, the total number of vertices in the two \((m+1)-\)gons is \[2\cdot (m+1) - 1 = 2m + 1 = n.\] 
Also observe that as \(n\) is odd, \(\lfloor\frac{m}{2}\rfloor = \lfloor\frac{n-1}{4}\rfloor = \lfloor\frac{n}{4}\rfloor\) of the distances occur \[2\cdot (m+1) = 2m + 2 = n +1\] times. Hence, we get a set of \(n\) points \(\{v_1,\dots,v_{m+1}, r_2,\dots,r_{m+1}\}\) where \(\lfloor \frac{n}{4}\rfloor\) distances occur at least \(n+1\) times. This concludes Case 1.

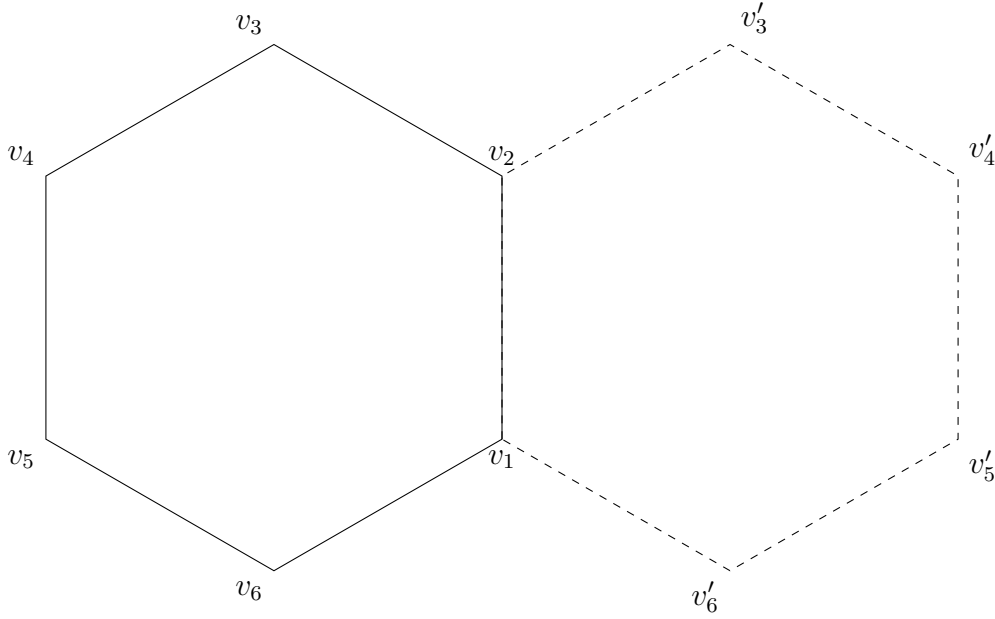
\begin{figure}[ht]
\centering
\begin{tikzpicture}[scale=3.5]

\def\side{1}

\coordinate (v1) at (0,0);
\coordinate (v2) at (0,1);
\coordinate (v3) at ({-sqrt(3)/2}, {1.5});
\coordinate (v4) at ({-sqrt(3)}, 1);
\coordinate (v5) at ({-sqrt(3)}, 0);
\coordinate (v6) at ({-sqrt(3)/2}, {-0.5});

\coordinate (v1') at (0,0);
\coordinate (v2') at (0,1);
\coordinate (v3') at ({sqrt(3)/2}, {1.5});
\coordinate (v4') at ({sqrt(3)}, 1);
\coordinate (v5') at ({sqrt(3)}, 0);
\coordinate (v6') at ({sqrt(3)/2}, {-0.5});

\draw (v1) -- (v2) -- (v3) -- (v4) -- (v5) -- (v6) -- cycle;

\draw[dashed] (v1') -- (v2') -- (v3') -- (v4') -- (v5') -- (v6') -- cycle;

\node[below] at (v1) {$v_1$};
\node[above] at (v2) {$v_2$};
\node[above left] at (v3) {$v_3$};
\node[above left] at (v4) {$v_4$};
\node[below left] at (v5) {$v_5$};
\node[below left] at (v6) {$v_6$};

\node[above right] at (v3') {$v'_3$};
\node[above right] at (v4') {$v'_4$};
\node[below right] at (v5') {$v'_5$};
\node[below left] at (v6') {$v'_6$};

\end{tikzpicture}
\caption{A set of 10 points with 2 distance appearing at least 11 times.}
\end{figure}

\textbf{Case 2}
     The proof is similar to that of Case 1. However, instead of rotating a regular polygon around one of its vertices, we will reflect a regular polygon on one of its edges. 

    Since \(n\) is even, let \(n = 2m\). Consider an \((m+1)-\)gon  \(v_1\dots v_{m+1}\).   From Claim \ref{claim:mgon}, \(\lfloor\frac{m}{2}\rfloor\) of the distances are repeated \(m+1\) times. Now, reflect the \((m+1)-\)gon over the edge \(v_1v_2\) to get a new \((m+1)-\)gon  \(v_1 v_2 v'_3\dots v'_{m+1}\). Again, \(\lfloor\frac{m}{2}\rfloor\) of the distances are repeated \(m+1\) times in the new \((m+1)-\)gon. Since \(v_1\) and \(v_2\) are the only common vertices between the two \((m+1)-\)gons, the total number of vertices in the union of the two \((m+1)-\)gons is \[2\cdot (m+1) - 2 = 2m  = n.\] Also, observe that the only distance common between the two \((m+1)-\)gons is \(\|v_1 - v_{2}\|\), and repeating only for the edge \(v_1v_2\). Thus, \(\lfloor\frac{m}{2}\rfloor = \lfloor\frac{n}{4}\rfloor\) of the distances occur at least \[2\cdot (m+1) - 1= 2m + 1 = n +1\] times. Hence, we get a set of \(n\) points \(\{v_1,\dots,v_{m+1}, v'_3,\dots,v'_{m+1}\}\) where \(\lfloor \frac{n}{4}\rfloor\) many of the distances occur at least \(n+1\) times. This concludes Case 2 and proves the theorem.

\end{proof}

\section{Proof of Theorem ~\ref{thm:richdistance2}}

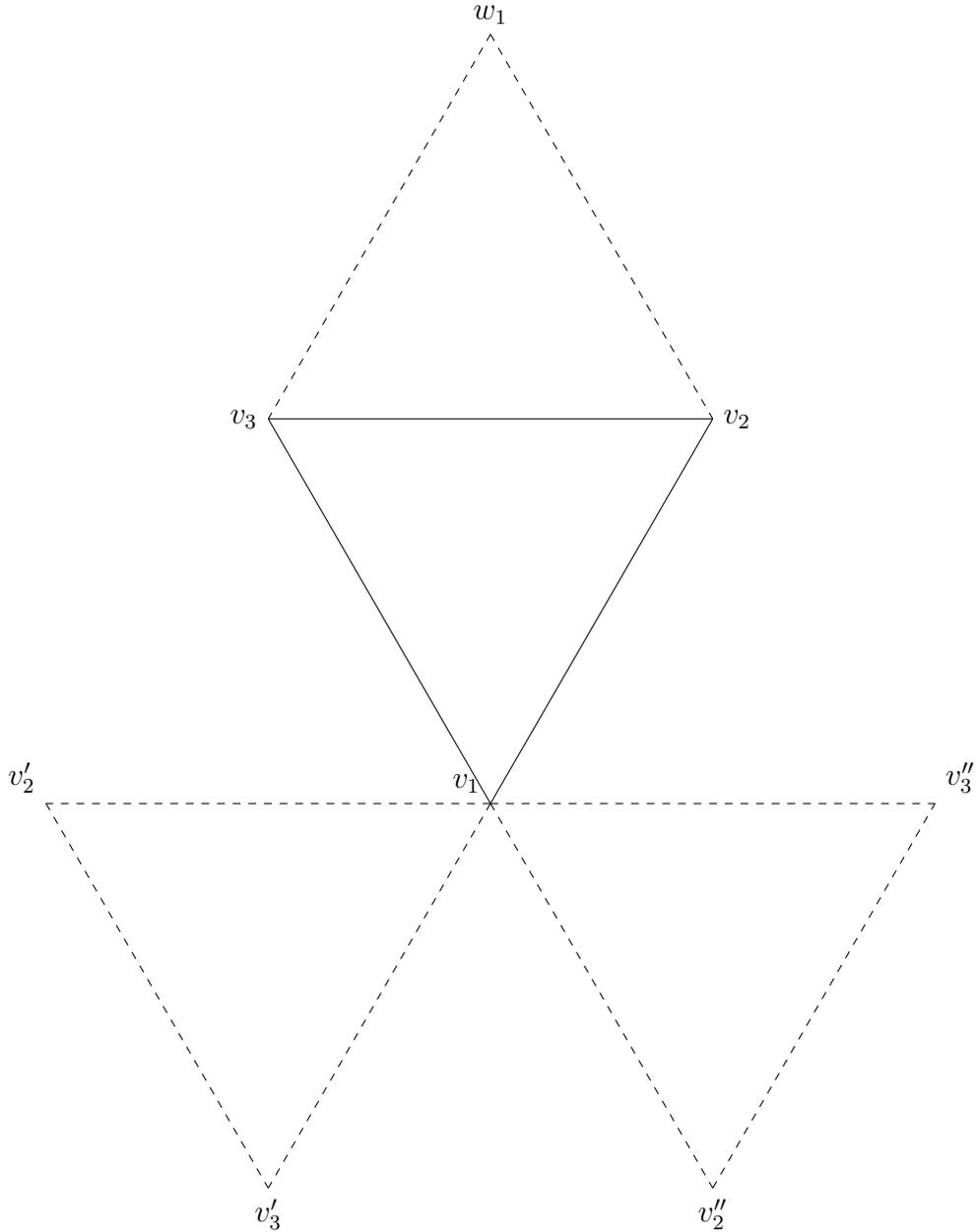
\begin{figure}[ht]
\centering
\begin{tikzpicture}[scale=3]

\foreach \i in {1,...,6}
    \coordinate (v\i) at ({60*(\i-1)}:2);

\coordinate (mid23) at ($(v2)!0.5!(v3)$);

\coordinate (w1) at ($2*(mid23) - (0,0)$);

\node[right] at (v2) {$v_2$};
\node[left] at (v3) {$v_3$};
\node[above left] at (v4) {$v'_2$};
\node[below] at (v5) {$v'_3$};
\node[below] at (v6) {$v''_2$};
\node[above right] at (v1) {$v''_3$};
\node[anchor=south east] at (0,0) {$v_1$};
\node[above] at (w1) {$w_1$};

\draw (v2) -- (v3);
\draw[dashed] (v4) -- (v5);
\draw[dashed] (v6) -- (v1);
\draw[dashed] (0,0) -- (v4);
\draw[dashed] (0,0) -- (v5);
\draw[dashed] (0,0) -- (v6);
\draw[dashed] (0,0) -- (v1);
\draw (0,0) -- (v2);
\draw (0,0) -- (v3);

\draw[dashed] (w1) -- (v2);
\draw[dashed] (w1) -- (v3);
\end{tikzpicture}
\caption{A set of 8 points with 1 distance appearing at least 11 times. The diagram consists of two rotations and one reflection of the triangle \(v_1v_2v_3\).}
\end{figure}

\begin{proof}[Proof of Theorem ~\ref{thm:richdistance2}] 
For \(n < m+3\) the statement is vacuously true hence, we will assume \(n \geq m+3\). Let \(n = (m+1)k + r\) for some \(r \in [2,m+2]\). To prove this theorem, we start with a regular \((k+2)-\)gon \(v_1\dots v_{k+2}\). Fix a vertex, say \(v_1\), and take \((r-2)\) arbitrary rotations of the \((k+2)-\)gon around  \(v_1\), resulting in a total of \((r-1)\) regular \((k+2)-\)gons with a common vertex \(v_1\). Now iteratively reflect the \((k+2)-\)gon over an edge \(v_i,v_{i+1}\) for some \(i\in [k]\), then chose another edge of any the \((k+2)-\)gon and reflect again with a total of \((m+2-r)\) reflections.   Hence, the total number of points is \[(k+2)+(r-2)(k+1)+(m+2-r)(k) = (m+1)k+r = n.\]

Observe that from Claim \ref{claim:mgon} \(\lfloor\frac{k+1}{2}\rfloor\) of the distances repeat \(k+2\) times in each \((k+2)-\)gon, with only repetition of one edge for each reflection.  Since there are \(m+2-r\) reflections in total, \(\lfloor\frac{k+1}{2}\rfloor\) distances appear at least \[(k+2)(m+1)-(m+2-r) = [(m+1)k+r]+m = n + m\]
times. Finally, since \(n = (m+1)k + r\) we have, \[\Big\lfloor\frac{k+1}{2}\Big\rfloor \geq \Big\lfloor\frac{n}{2(m+1)}\Big\rfloor,\] and we conclude that at least \(\big\lfloor\frac{n}{2(m+1)}\big\rfloor\) of the distances appear at least \(n+m\) times. 
\end{proof}

\section{Further Research}
In \cite{p}, the main problem of this paper was mentioned in the context of the previously conjectured Erdős' distinct distance problem \cite{pe} from 1946.

\begin{Problem}[Erdős' Distinct Distance Problem]
    Does every set of \(n\) distinct points in \(\mathbb{R}^2\) determine \(\gg n/\sqrt{\log n}\) many distinct distances?
\end{Problem}

Erdős' distinct distance problem was almost settled (with a remaining gap of \(\sqrt{\log{n}}\)) by Guth and Katz \cite{gk}. In the same paper \cite{p} Erdős also mentioned another question of himself and Pach.

\begin{Problem}[Erdős and Pach]
    Let \(A\subset \mathbb{R}^2\) be a set of \(n\) points. Must there be two distances which occur at least once but between at most \(n\) pairs of points?
\end{Problem}

Pannwitz and Hopf \cite{hp} proved that the largest distance between points of \(A\) can occur at most \(n\) times, but it remains unknown whether a second such distance must occur. Erdős and Pach believe that such a distance exists.

Another popular distance problem of Erdős mentioned alongside the distinct distance problem \cite{pe} is known as Erdős' unit distance problem.

\begin{Problem}[Erdős' Unit Distance Problem]
    Does every set of \(n\) distinct points in \(\mathbb{R}^2\) contain at most \(n^{1+O(1/\log\log n)}\) pairs which are distance \(1\) apart?
\end{Problem}

This bound would be best possible as it is achievable for the integer lattice. The best known upper bound is \(O(n^{4/3})\), due to Spencer, Szemerédi, and Trotter \cite{sst}.

\section*{Acknowledgements}
The author was supported by the Austrian Science Fund FWF Project P 34180. I am deeply grateful to Oliver Roche-Newton for his invaluable comments and suggestions. I also thank Alexander Polyanskii and Adam Sheffer for their feedback on an earlier draft of this paper.

\end{document}